\RequirePackage{fix-cm}
\documentclass[smallextended, envcountsame]{svjour3}
\smartqed
\journalname{}

\usepackage[unicode = false,
    pdftoolbar = true,
    colorlinks = true,
    linkcolor = blue,
    citecolor = blue,
    filecolor = black,
    urlcolor = blue,
    breaklinks = true]{hyperref}

\usepackage[utf8]{inputenc}
\usepackage{listings}  
\usepackage{enumitem}  
\usepackage{amsmath}   
\usepackage{amssymb}   
\usepackage{xcolor}    
\usepackage{graphicx}  
\usepackage{verbatim}
\usepackage{geometry}
\usepackage{mathtools}
\usepackage{float}
\usepackage{bm}
\usepackage[ruled]{algorithm2e}
\usepackage{subcaption}
\usepackage{placeins}
\usepackage{listings}
\usepackage[justification=centering]{caption}
\usepackage{multirow}
\usepackage{orcidlink}
\usepackage{bbding}
\usepackage{amsthm}

\definecolor{gab}{HTML}{50c878}
\definecolor{pat}{HTML}{bf68ff}
\definecolor{sham}{HTML}{7faaff}

\newtheorem{assump}{Assumption}
\newtheorem{coroll}{Corollary}
\newtheorem{defini}{Definition}
\newtheorem{exampl}{Example}

\newtheorem{rema}{Remark}
\newtheorem{theo}{Theorem}

\allowdisplaybreaks[4]

\DeclareMathAlphabet{\mymathbb}{U}{BOONDOX-ds}{m}{n}

\makeatletter
\newcommand*{\rom}[1]{\expandafter\@slowromancap\romannumeral #1@}
\makeatother

\SetKwIF{If}{ElseIf}{Else}{if}{}{else if}{else}{end if}%
\SetKwRepeat{Do}{do}{while}
\allowdisplaybreaks[4]
\SetArgSty{textnormal}

\title{Relationships between full-space and subspace quadratic interpolation models and simplex derivatives}

\author{Yiwen Chen\,\orcidlink{0000-0002-8720-932X}\,\Envelope}
\institute{Department of Mathematics, University of British Columbia, Kelowna, British Columbia, V1V 1V7, Canada.\\
This research is partially funded by the Killam Doctoral Scholarship and the UBC Okanagan Distinguished Doctoral Scholar Award.\\
\email{yiwchen@student.ubc.ca}}
\date{\today}

\begin{document}

\maketitle

\begin{abstract}
    Quadratic interpolation models and simplex derivatives are fundamental tools in numerical optimization, particularly in derivative-free optimization.  When constructed in suitably chosen affine subspaces, these tools have been shown to be especially effective for high-dimensional derivative-free optimization problems, where full-space model construction is often impractical.  In this paper, we analyze the relationships between full-space and subspace formulations of these tools.  In particular, we derive explicit conversion formulas between full-space and subspace models, including minimum-norm models, minimum Frobenius norm models, least Frobenius norm updating models, as well as models constructed via generalized simplex gradients and Hessians.  We show that the full-space and subspace models coincide on the affine subspace and, in general, along directions in the orthogonal complement.  Overall, our results provide a theoretical framework for understanding subspace approximation techniques and offer insight into the design and analysis of derivative-free optimization methods.
\end{abstract}
\medskip


\section{Introduction}
    Derivative-free optimization (DFO) methods are a class of optimization methods designed for problems where the derivatives of the objective and/or constraint functions are unavailable or expensive to compute~\cite{audet2017derivative,conn2009introduction}.  One major class of DFO methods, known as model-based DFO methods, constructs approximation models from function evaluations and utilizes these models to guide the optimization process.  Among the most widely used modeling and approximation techniques are quadratic interpolation models (see, e.g., \cite[Chapters 2 to 6]{conn2009introduction}) and simplex derivatives, which include generalized simplex gradients \cite{conn2008geometry,custodio2008using,custodio2007using,regis2015calculus} and Hessians \cite{hare2024matrix}.  A general framework for the approximation errors of first-order simplex derivatives can be found in~\cite{chen2025general}.  These techniques have been the backbone of many model-based DFO methods \cite{audet2020model,berahas2019derivative,larson2019derivative,liuzzi2019trust}.

    As the problem dimension grows, however, the construction of accurate full-space models becomes increasingly challenging.  Indeed, a determined quadratic interpolation model (see Definition \ref{def:detquad}) in~$\mathbb{R}^n$ requires $(n+1)(n+2)\slash 2$ sample points \cite[Chapter 3]{conn2009introduction}, which scales quadratically with the dimension $n$.  This challenge has motivated the development of subspace model-based DFO methods, in which models are constructed in carefully chosen low-dimensional affine subspaces.  These methods have proven, both numerically and theoretically, to be promising for handling high-dimensional DFO problems \cite{cartis2023scalable,cartis2026randomized,chen2024qfully,chen2025clarsta,dzahini2024stochastic}.

    Despite the successful application of subspace approximations, the theoretical relationship between subspace models and their full-space counterparts remains only partially understood.  In this paper, we provide a theoretical framework that clarifies these relationships.  Assuming that the sample set lies in an affine subspace, we derive conversion formulas between full-space and subspace formulations of several widely used quadratic interpolation models, including minimum-norm models \cite{conn1996algorithm}, minimum Frobenius norm models~\cite{conn1998derivative}, and least Frobenius norm updating models \cite{powell2004least}.  We show that the full-space and subspace models coincide on the affine subspace and, in general, along directions in the orthogonal complement.  We also establish corresponding conversion formulas and coincidence results for generalized simplex gradients \cite{conn2008geometry,custodio2008using,custodio2007using,regis2015calculus} and generalized simplex Hessians \cite{hare2024matrix}, as well as the models constructed from them.  Using the same analysis technique, similar results can be achieved for other models and simplex derivatives.  Overall, these results deepen the theoretical understanding of subspace approximation techniques and provide insight into the design and analysis of model-based DFO methods.

    The remainder of this paper is structured as follows.  We end this section by introducing the notations and background definitions used in this paper.  Section \ref{sec:full&subspaceUnderdModels} analyzes the relationships between full-space and subspace minimum norm, minimum Frobenius norm, and least Frobenius norm updating models.  Section \ref{sec:full&subspaceGSGGSH} analyzes the relationships between full-space and subspace generalized simplex gradients, Hessians, and the models constructed from them.  Section~\ref{sec:conclusion} concludes the work covered by this paper and suggests some promising future directions.

    \subsection{Notations}
        We use $f:\mathbb{R}^n\to\mathbb{R}$ to denote the function to be approximated.  We use $\mymathbb{0}$ to denote the all-zero vector with appropriate dimensions.  We use $e^i$ to denote the $i$-th coordinate vector with appropriate dimensions.  We use $I_n$ to denote the identity matrix of order $n$.  We use $S_n(\mathbb{R})$ to denote the set of all $n\times n$ real symmetric matrices.  We use $\mathrm{col}(A)$ to denote the column space of matrix $A$ and $\mathrm{col}(A)^\perp$ to denote its orthogonal complement.  We use $\|\cdot\|$ to denote the Euclidean norm of a vector and the spectral norm of a matrix.  We use $\|\cdot\|_F$ to denote the Frobenius norm of a matrix.

    \subsection{Determined quadratic interpolation, minimum norm, minimum Frobenius norm, and least Frobenius norm updating models}
        We say that a sample set $Y\subseteq\mathbb{R}^n$ is poised for quadratic interpolation if the linear system associated with the interpolation conditions is determined \cite[Chapter 3]{conn2009introduction}.  In this case, there exists a unique quadratic function that interpolates $f$ on $Y$. We refer to this quadratic function as the {\em determined quadratic interpolation (DQI) model}.
        \begin{defini}\label{def:detquad}
            Let $x^0\in\mathbb{R}^n$ be the point of interest.  Suppose that $Y$ is poised for quadratic interpolation.  The determined quadratic interpolation model of $f$ at $x^0$ over $Y$, denoted by $m_{\mathrm{DQI}}^{x^0,Y}(x)$ is the unique quadratic function satisfying
            \begin{equation*}
                m_{\mathrm{DQI}}^{x^0,Y}(y) = f(y)~~~\text{for all $y\in Y$}.
            \end{equation*}
        \end{defini}

        If the associated linear system is underdetermined, there are infinitely many quadratic functions that interpolate $f$ on $Y$.  In this case, additional criteria are required to select a quadratic interpolation model. Several such approaches have been proposed in the literature, and the most widely used ones are the {\em minimum norm (MN) models} \cite{conn1996algorithm}, {\em minimum Frobenius norm (MFN) models} \cite{conn1998derivative}, and {\em least Frobenius norm updating (LFU) models} \cite{powell2004least}.
        \begin{defini}
            Let $x^0\in\mathbb{R}^n$ be the point of interest.  Suppose that 
            \begin{equation*}
                Y=\{x^0\}\cup\{x^0+d^i:i=1,\ldots,m\}.
            \end{equation*}
            The class of minimum norm models of $f$ at $x^0$ over $Y$ is
            \begin{equation*}
                \left\{m_{\mathrm{MN}}^{x^0,Y}(x) = f(x^0) + \nabla_{\mathrm{MN}}f(x^0;Y)^\top\left(x-x^0\right) + \frac{1}{2}\left(x-x^0\right)^\top \nabla_{\mathrm{MN}}^2f(x^0;Y)\left(x-x^0\right)\right\}
            \end{equation*}
            where $(\nabla_{\mathrm{MN}}f(x^0;Y),\nabla_{\mathrm{MN}}^2f(x^0;Y))\in\mathcal{M}^{x^0,Y}_{\mathrm{MN}}$ and $\mathcal{M}^{x^0,Y}_{\mathrm{MN}}$ is the set of minimizers of
            \begin{align}
                \begin{split}\label{pro:MN}
                    \min_{\substack{\alpha\in\mathbb{R}^n\\ H\in S_n(\mathbb{R})}}~~~&\frac{1}{2}\left\|\alpha\right\|^2+\frac{1}{2}\left\|H\right\|_F^2\\
                    s.t.~~~&\left(d^i\right)^\top\alpha+\frac{1}{2}\left(d^i\right)^\top H\left(d^i\right) = f(x^0+d^i)-f(x^0),~~~i=1,\ldots,m.
                \end{split}\tag{MN}
            \end{align}
        \end{defini}
        \begin{rema}\label{rem:MN_unique}
            In Problem \eqref{pro:MN}, the objective function is strongly convex, and all constraints are linear with respect to $\alpha$ and $H$.  The minimizer of Problem~\eqref{pro:MN} exists and is unique if and only if Problem~\eqref{pro:MN} is feasible, i.e., there exists a quadratic function that interpolates $f$ on $Y$.  
        \end{rema}
        
        \begin{defini}
            Let $x^0\in\mathbb{R}^n$ be the point of interest.  Suppose that 
            \begin{equation*}
                Y=\{x^0\}\cup\{x^0+d^i:i=1,\ldots,m\}.
            \end{equation*}
            The class of minimum Frobenius norm models of $f$ at $x^0$ over $Y$ is
            \begin{equation*}
                \left\{m_{\mathrm{MFN}}^{x^0,Y}(x) = f(x^0) + \nabla_{\mathrm{MFN}}f(x^0;Y)^\top\left(x-x^0\right) + \frac{1}{2}\left(x-x^0\right)^\top\nabla_{\mathrm{MFN}}^2f(x^0;Y)\left(x-x^0\right)\right\}
            \end{equation*}
            where $(\nabla_{\mathrm{MFN}}f(x^0;Y),\nabla_{\mathrm{MFN}}^2f(x^0;Y))\in\mathcal{M}^{x^0,Y}_{\mathrm{MFN}}$ and $\mathcal{M}^{x^0,Y}_{\mathrm{MFN}}$ is the set of minimizers of
            \begin{align}
                \begin{split}\label{pro:MFN}
                    \min_{\substack{\alpha\in\mathbb{R}^n\\ H\in S_n(\mathbb{R})}}~~~&\frac{1}{2}\left\|H\right\|_F^2\\
                    s.t.~~~&\left(d^i\right)^\top\alpha+\frac{1}{2}\left(d^i\right)^\top H\left(d^i\right) = f(x^0+d^i)-f(x^0),~~~i=1,\ldots,m.
                \end{split}\tag{MFN}
            \end{align}
        \end{defini}
        \begin{rema}\label{rem:MFN_unique}
            If we only consider $H$ as a variable, then Problem \eqref{pro:MFN} can be written as $$\min_{H\in P}\frac{1}{2}\left\|H\right\|_F^2$$ where
            \begin{equation*}
                P=\left\{H\in S_n(\mathbb{R}):\exists\alpha\in\mathbb{R}^n~s.t.~\frac{1}{2}\left(d^i\right)^\top H\left(d^i\right) = f(x^0+d^i)-f(x^0)-\left(d^i\right)^\top\alpha,~i=1,\ldots,m\right\},
            \end{equation*}
            i.e., minimizing a strongly convex function over a closed convex set.  We have that $\nabla_{\mathrm{MFN}}^2f(x^0;Y)$ exists and is unique if and only if Problem \eqref{pro:MFN} is feasible.  Moreover, $\nabla_{\mathrm{MFN}}f(x^0;Y)$ exists (not necessarily unique) if and only if Problem \eqref{pro:MFN} is feasible.  
        \end{rema}
        The following example shows that, when Problem \eqref{pro:MFN} is feasible, $\nabla_{\mathrm{MFN}}f(x^0;Y)$ exists but may not be unique.
        \begin{exampl}\label{exp:MFNgradnotunique}
            Consider $n=3$, $f(x)=\|x\|^2$, $x^0=\mymathbb{0}$ and $Y=\{\mymathbb{0},e^1,e^2,e^1+e^2\}\subseteq\mathbb{R}^3$.  Then, the constraints of Problem~\eqref{pro:MFN} are 
            \begin{align*}
            \begin{cases}
                \left(e^1\right)^\top\alpha+\frac{1}{2}\left(e^1\right)^\top H\left(e^1\right) = 1\\
                \left(e^2\right)^\top\alpha+\frac{1}{2}\left(e^2\right)^\top H\left(e^2\right) = 1\\
                \left(e^1+e^2\right)^\top\alpha+\frac{1}{2}\left(e^1+e^2\right)^\top H\left(e^1+e^2\right) = 2.
            \end{cases}
            \end{align*}
            Solving Problem~\eqref{pro:MFN}, we get
            \begin{align*}
                \mathcal{M}^{x^0,Y}_{\mathrm{MFN}} = \left\{\left(\begin{bmatrix}
                    1\\
                    1\\
                    r
                \end{bmatrix}, \begin{bmatrix}
                    0 & 0 & 0\\
                    0 & 0 & 0\\
                    0 & 0 & 0
                \end{bmatrix}\right): r\in\mathbb{R}\right\},
            \end{align*}
            which implies that $\nabla_{\mathrm{MFN}}f(x^0;Y)$ exists but is not unique.
        \end{exampl}

        For simplicity, when Problem \eqref{pro:MFN} is feasible, we define
        \begin{equation*}
            \mathcal{G}^{x^0,Y}_{\mathrm{MFN}} = \left\{\nabla_{\mathrm{MFN}} f(x^0;Y):\left(\nabla_{\mathrm{MFN}} f(x^0;Y), \nabla^2_{\mathrm{MFN}} f(x^0;Y)\right) \in \mathcal{M}^{x^0,Y}_{\mathrm{MFN}}\right\}.
        \end{equation*}
    
        \begin{defini}
            Let $x^0\in\mathbb{R}^n$ be the point of interest and $\overline{H}\in S_n(\mathbb{R})$.  Suppose that 
            \begin{equation*}
                Y=\{x^0\}\cup\{x^0+d^i:i=1,\ldots,m\}.
            \end{equation*}
            The class of least Frobenius norm updating models of $f$ at $x^0$ over $Y$ with respect to $\overline{H}$ is
            \begin{equation*}
                \left\{m_{\mathrm{LFU}}^{x^0,Y,\overline{H}}(x) = f(x^0) + \nabla_{\mathrm{LFU}}f(x^0;Y;\overline{H})^\top\left(x-x^0\right) + \frac{1}{2}\left(x-x^0\right)^\top\nabla_{\mathrm{LFU}}^2f(x^0;Y;\overline{H})\left(x-x^0\right)\right\}
            \end{equation*}
            where $(\nabla_{\mathrm{LFU}}f(x^0;Y;\overline{H}),\nabla_{\mathrm{LFU}}^2f(x^0;Y;\overline{H}))\in\mathcal{M}_{\mathrm{LFU}}^{x^0,Y,\overline{H}}$ and $\mathcal{M}_{\mathrm{LFU}}^{x^0,Y,\overline{H}}$ is the set of minimizers of
            \begin{align}
                \begin{split}\label{pro:LFU}
                    \min_{\substack{\alpha\in\mathbb{R}^n\\ H\in S_n(\mathbb{R})}}~~~&\frac{1}{2}\left\|H-\overline{H}\right\|_F^2\\
                    s.t.~~~&\left(d^i\right)^\top\alpha+\frac{1}{2}\left(d^i\right)^\top H\left(d^i\right) = f(x^0+d^i)-f(x^0),~~~i=1,\ldots,m.
                \end{split}\tag{LFU}
            \end{align}
        \end{defini}
        \begin{rema}
            In practice, $\overline{H}$ is often set to be a reasonable approximation of $\nabla^2f(x^0)$, e.g., the model Hessian from the previous iteration.
        \end{rema}
        \begin{rema}\label{rem:LFU_unique}
            Similar to the MFN models, $\nabla_{\mathrm{LFU}}^2f(x^0;Y;\overline{H})$ exists and is unique if and only if Problem \eqref{pro:LFU} is feasible.  Moreover, $\nabla_{\mathrm{LFU}}f(x^0;Y;\overline{H})$ exists (not necessarily unique) if and only if Problem \eqref{pro:LFU} is feasible. 
        \end{rema}
        Notice that when $\overline{H}$ is the zero matrix, Problem \eqref{pro:LFU} reduces to Problem \eqref{pro:MFN}.  Hence, Example \ref{exp:MFNgradnotunique} also illustrates that, when Problem \eqref{pro:LFU} is feasible, $\nabla_{\mathrm{LFU}}f(x^0;Y;\overline{H})$ exists but may not be unique.  For simplicity, when Problem \eqref{pro:LFU} is feasible, we define
        \begin{equation*}
            \mathcal{G}^{x^0,Y,\overline{H}}_{\mathrm{LFU}} = \left\{\nabla_{\mathrm{LFU}} f(x^0;Y;\overline{H}):\left(\nabla_{\mathrm{LFU}} f(x^0;Y;\overline{H}), \nabla^2_{\mathrm{LFU}} f(x^0;Y;\overline{H})\right) \in \mathcal{M}^{x^0,Y,\overline{H}}_{\mathrm{LFU}}\right\}.
        \end{equation*}

        From the definitions, it follows that if $Y$ is poised for quadratic interpolation, then the MN, MFN, and LFU models all coincide with the DQI model.

\subsection{Generalized simplex gradient, generalized simplex Hessian, and quadratic generalized simplex derivative models}
    Simplex derivatives are often used to approximate the derivatives of the objective function and to construct approximation models.  We examine two of the most widely used simplex derivatives in this paper, namely the {\em generalized simplex gradient (GSG)} \cite{conn2008geometry,custodio2008using,custodio2007using,regis2015calculus} and {\em generalized simplex Hessian (GSH)} \cite{hare2024matrix}.
    
    \begin{defini}
        Let $x^0\in\mathbb{R}^n$ be the point of interest.  Let matrix $S=[s^1\cdots s^p]\in\mathbb{R}^{n\times p}$.  The generalized simplex gradient of $f$ at $x^0$ over $S$ is
        \begin{equation*}
            \nabla_sf(x^0;S) = \left(S^\top\right)^\dagger\delta_f(x^0;S)
        \end{equation*}
        where
        \begin{equation*}
            \delta_f(x^0;S) = \begin{bmatrix}
                f(x^0+s^1) - f(x^0)\\
                \vdots\\
                f(x^0+s^p) - f(x^0)
            \end{bmatrix}\in\mathbb{R}^p.
        \end{equation*}
    \end{defini}
    
    \begin{defini}
        Let $x^0\in\mathbb{R}^n$ be the point of interest.  Let matrices $S=[s^1\cdots s^p]\in\mathbb{R}^{n\times p}$ and $T_i=[t_i^1\cdots t_i^{q_i}]\in\mathbb{R}^{n\times q_i}$ where $i=1,\ldots,p$.  Denote $T_{1:p}=\{T_1,\ldots,T_p\}$.  The generalized simplex Hessian of $f$ at $x^0$ over $S$ and $T_{1:p}$ is
        \begin{equation*}
            \nabla_s^2f(x^0;S;T_{1:p}) = \left(S^\top\right)^\dagger\delta_f^2(x^0;S;T_{1:p})
        \end{equation*}
        where
        \begin{equation*}
            \delta_f^2(x^0;S;T_{1:p}) = \begin{bmatrix}
                \left(\nabla_sf(x^0+s^1;T_1)-\nabla_sf(x^0;T_1)\right)^\top\\
                \vdots\\
                \left(\nabla_sf(x^0+s^p;T_p)-\nabla_sf(x^0;T_p)\right)^\top
            \end{bmatrix}\in\mathbb{R}^{p\times n}.
        \end{equation*}
    \end{defini}
    \begin{rema}
        Following \cite{hare2024matrix}, if $T_1=\cdots=T_p=T$, then, with a little misuse of notation, we denote $T_{1:p}=T$. 
    \end{rema}
    
    There are several ways to construct quadratic models using the GSG and GSH.  In this paper, we define the {\em quadratic generalized simplex derivative (QGSD) models} as follows. Unlike the MN, MFN, and LFU models, the QGSD models are constructed using linear algebraic operations rather than by solving optimization problems. 
    \begin{defini}
        Let $x^0\in\mathbb{R}^n$ be the point of interest.  Let 
        \begin{align*}
            Y=\left\{x^0\right\}&\cup\left\{x^0+s^i:i=1,\ldots,p\right\}\\
            &\cup\left\{x^0+t_i^j:i=1,\ldots,p,j=1,\ldots,q_i\right\}\cup\left\{x^0+s^i+t_i^j:i=1,\ldots,p,j=1,\ldots,q_i\right\}.
        \end{align*}
        The class of quadratic generalized simplex derivative models of $f$ at $x^0$ over $Y$ is
        \begin{align*}
            \left\{m_{\mathrm{QGSD}}^{x^0,Y}(x) = f(x^0)+\nabla_{\mathrm{QGSD}}f(x^0;Y)^\top\left(x-x^0\right) +\frac{1}{2}\left(x-x^0\right)^\top\nabla^2_{\mathrm{QGSD}}f(x^0;Y)\left(x-x^0\right)\right\}
        \end{align*}
        where $\nabla_{\mathrm{QGSD}}f(x^0;Y)$ and $\nabla^2_{\mathrm{QGSD}}f(x^0;Y)$ are linear combinations of GSG and GSH constructed using points in $Y$, chosen such that the union of the points used in both constructions equals $Y$.
    \end{defini}
    \begin{rema}
        The simplest choice of $(\nabla_{\mathrm{QGSD}}f(x^0;Y),\nabla^2_{\mathrm{QGSD}}f(x^0;Y))$ is
        \begin{equation*}
            \nabla_{\mathrm{QGSD}}f(x^0;Y) = \nabla_sf(x^0;S)~~~\text{and}~~~\nabla^2_{\mathrm{QGSD}}f(x^0;Y) = \nabla^2_sf(x^0;S;T_{1:p}),
        \end{equation*}
        where $\nabla^2_{\mathrm{QGSD}}f(x^0;Y)$ can be replaced by $\frac{1}{2}(\nabla^2_sf(x^0;S;T_{1:p})+\nabla^2_sf(x^0;S;T_{1:p})^\top)$ if a symmetric model Hessian is desired.
        
        An alternative choice is studied in \cite{chen2024qfully}, where the authors assume $T_{1:p}=S$ and let
        \begin{equation}\label{eq:qgsdexp}
            \nabla_{\mathrm{QGSD}}f(x^0;Y) = 2\nabla_sf(x^0;S)-\nabla_sf(x^0;2S)~~~\text{and}~~~\nabla^2_{\mathrm{QGSD}}f(x^0;Y) = \nabla^2_sf(x^0;S;T_{1:p}).
        \end{equation}
        It is shown that when $S$ has full column rank, the QGSD model corresponding to \eqref{eq:qgsdexp} is an underdetermined quadratic interpolation model of $f$ at $x^0$ over $Y$; see \cite[Theorem~1]{chen2024qfully} for details.
    \end{rema}

\section{Full-space and subspace MN, MFN, and LFU models}\label{sec:full&subspaceUnderdModels}
    In this section, we explore the relationships between the full-space and subspace MN, MFN, and LFU models, and we derive conversion formulas for each subspace model and its corresponding full-space counterpart.  

    We shall use the following assumptions.
    \begin{assump}\label{ass:Yisinsubspace}
        Suppose that $Y=\{x^0\}\cup\{x^0+d^i:i=1,\ldots,m\}$ lies in the $d$-dimensional affine space $\mathcal{Y}=\{x^0+Q\widehat{x}:\widehat{x}\in\mathbb{R}^d\}$ where $d<n$ and $Q\in\mathbb{R}^{n\times d}$ has $d$ orthonormal columns.  
    \end{assump}
    \begin{assump}\label{ass:interpexists}
        Suppose that there exists a quadratic function that interpolates $f$ on $Y$.
    \end{assump}
    Under Assumption \ref{ass:Yisinsubspace}, for each $d^i$, there exists $\widehat{d}^i\in\mathbb{R}^d$ such that $d^i=Q\widehat{d}^i$.  We denote $\widehat{Y}=\{\mymathbb{0}\}\cup\{\widehat{d}^i:i=1,\ldots,m\}$ and define $\widehat{f}(\widehat{x})=f(x^0+Q\widehat{x})$.  Assumption \ref{ass:interpexists} is equivalent to assuming that there exists a quadratic function that interpolates $\widehat{f}$ on $\widehat{Y}$.  Throughout this paper, we refer to the model of $f$ over $Y$ as the full-space model, and the model of $\widehat{f}$ over $\widehat{Y}$ as the subspace model.

    We begin by analyzing the relationships between the full-space and subspace MN models.  We show that the two models coincide on $\mathcal{Y}$ and along directions in $\mathrm{col}(Q)^\perp$.  Notice that by Remark~\ref{rem:MN_unique}, under Assumptions~\ref{ass:Yisinsubspace} and \ref{ass:interpexists}, both $(\nabla_{\mathrm{MN}}f(x^0;Y),\nabla_{\mathrm{MN}}^2f(x^0;Y))$ and $(\nabla_{\mathrm{MN}}\widehat{f}(\mymathbb{0};\widehat{Y}),\nabla_{\mathrm{MN}}^2\widehat{f}(\mymathbb{0};\widehat{Y}))$ exist and are unique. 
    \begin{theo}\label{thm:full&subspaceMN}
         Suppose that Assumptions~\ref{ass:Yisinsubspace} and \ref{ass:interpexists} hold. Then, both $(\nabla_{\mathrm{MN}}f(x^0;Y),\nabla_{\mathrm{MN}}^2f(x^0;Y))$ and $(\nabla_{\mathrm{MN}}\widehat{f}(\mymathbb{0};\widehat{Y}),\nabla_{\mathrm{MN}}^2\widehat{f}(\mymathbb{0};\widehat{Y}))$ are unique and satisfy
        \begin{equation*}
            \left(\nabla_{\mathrm{MN}}f(x^0;Y), \nabla_{\mathrm{MN}}^2f(x^0;Y)\right) = \left(Q\nabla_{\mathrm{MN}}\widehat{f}(\mymathbb{0};\widehat{Y}), Q\nabla_{\mathrm{MN}}^2\widehat{f}(\mymathbb{0};\widehat{Y})Q^\top\right).
        \end{equation*}
        Moreover, 
        \begin{equation}\label{eq:samemodelvalMN}
            m_{\mathrm{MN}}^{x^0,Y}(x)=\widehat{m}_{\mathrm{MN}}^{\mymathbb{0},\widehat{Y}}(\widehat{x})~~~\text{for all $(x,\widehat{x})\in\mathbb{R}^n\times\mathbb{R}^d$ with $x\in \left(x^0+Q\widehat{x}\right)+\mathrm{col}(Q)^\perp$}.
        \end{equation}
    \end{theo}
    \begin{proof}
        Denote by $Q_{\perp}\in\mathbb{R}^{n\times (n-d)}$ a matrix with orthonormal columns spanning $\mathrm{col}(Q)^\perp$.  For all $\alpha\in\mathbb{R}^n$ and $H\in S_n(\mathbb{R})$, we can decompose
    \begin{equation*}
        \alpha=Q\widehat{\alpha}+Q_{\perp}\widehat{\alpha}_{\perp}~~~\text{and}~~~H=Q\widehat{H}Q^\top + Q\widehat{H}_0Q_{\perp}^\top + Q_{\perp}\widehat{H}_0^\top Q^\top + Q_{\perp}\widehat{H}_{\perp}Q_{\perp}^\top,
    \end{equation*}
    where $\widehat{\alpha}=Q^\top\alpha\in\mathbb{R}^d$, $\widehat{\alpha}_{\perp}=Q_{\perp}^\top\alpha\in\mathbb{R}^{n-d}$, $\widehat{H}=Q^\top HQ\in S_d(\mathbb{R})$, $\widehat{H}_0=Q^\top HQ_{\perp}\in\mathbb{R}^{d\times (n-d)}$, and $\widehat{H}_{\perp}=Q_{\perp}^\top HQ_{\perp}\in S_{n-d}(\mathbb{R})$.  We have that
    \begin{equation*}
        \left\|\alpha\right\|^2=\left\|\widehat{\alpha}\right\|^2+\left\|\widehat{\alpha}_{\perp}\right\|^2~~~\text{and}~~~\left\|H\right\|_F^2=\left\|\widehat{H}\right\|_F^2+2\left\|\widehat{H}_0\right\|_F^2+\left\|\widehat{H}_{\perp}\right\|_F^2
    \end{equation*}
    by the orthogonality of the columns of $Q$, and
    \begin{equation}\label{eq:fullsubsameconstraint}
        \left(d^i\right)^\top\alpha+\frac{1}{2}\left(d^i\right)^\top H\left(d^i\right) = \left(\widehat{d}^i\right)^\top Q^\top\alpha+\frac{1}{2}\left(\widehat{d}^i\right)^\top Q^\top HQ\left(\widehat{d}^i\right) = \left(\widehat{d}^i\right)^\top \widehat{\alpha}+\frac{1}{2}\left(\widehat{d}^i\right)^\top \widehat{H}\left(\widehat{d}^i\right).
    \end{equation}
    Therefore, Problem \eqref{pro:MN} becomes
    \begin{align*}
        \min_{\substack{\widehat{\alpha}\in\mathbb{R}^d,\widehat{H}\in S_d(\mathbb{R})\\ \widehat{\alpha}_{\perp}\in\mathbb{R}^{n-d}\\ \widehat{H}_0\in\mathbb{R}^{d\times (n-d)},\widehat{H}_{\perp}\in S_{n-d}(\mathbb{R})}}~~~&\frac{1}{2}\left\|\widehat{\alpha}\right\|^2+\frac{1}{2}\left\|\widehat{\alpha}_{\perp}\right\|^2+\frac{1}{2}\left\|\widehat{H}\right\|_F^2+\left\|\widehat{H}_0\right\|_F^2+\frac{1}{2}\left\|\widehat{H}_{\perp}\right\|_F^2\\
        s.t.~~~&\left(\widehat{d}^i\right)^\top \widehat{\alpha}+\frac{1}{2}\left(\widehat{d}^i\right)^\top \widehat{H}\left(\widehat{d}^i\right) = \widehat{f}(\widehat{d}^i)-\widehat{f}(\mymathbb{0}),~~~i=1,\ldots,m.
    \end{align*}
    Since the constraints do not contain $\widehat{\alpha}_\perp$, $\widehat{H}_0$, or $\widehat{H}_\perp$, these must be zero at the optimal solution.  Therefore, we have that
    \begin{equation*}
        \left(\nabla_{\mathrm{MN}}f(x^0;Y), \nabla_{\mathrm{MN}}^2f(x^0;Y)\right) = \left(Q\nabla_{\mathrm{MN}}\widehat{f}(\mymathbb{0};\widehat{Y}), Q\nabla_{\mathrm{MN}}^2\widehat{f}(\mymathbb{0};\widehat{Y})Q^\top\right).
    \end{equation*}

    Finally, let $x\in (x^0+Q\widehat{x})+\mathrm{col}(Q)^\perp$, i.e., $x=x^0+Q\widehat{x}+v$ for some $v\in\mathrm{col}(Q)^\perp$.  We have
    \begin{align*}
        m_{\mathrm{MN}}^{x^0,Y}(x) &= f(x^0)+\nabla_{\mathrm{MN}}f(x^0;Y)^\top \left(Q\widehat{x}+v\right)+\frac{1}{2}\left(\widehat{x}^\top Q^\top+v^\top\right)\nabla_{\mathrm{MN}}^2f(x^0;Y)\left(Q\widehat{x}+v\right)\\
        &= \widehat{f}(\mymathbb{0})+\nabla_{\mathrm{MN}}\widehat{f}(\mymathbb{0};\widehat{Y})^\top Q^\top \left(Q\widehat{x}+v\right)+\frac{1}{2}\left(\widehat{x}^\top Q^\top + v^\top\right)Q\nabla_{\mathrm{MN}}^2\widehat{f}(\mymathbb{0};\widehat{Y})Q^\top \left(Q\widehat{x}+v\right)\\
        &= \widehat{f}(\mymathbb{0})+\nabla_{\mathrm{MN}}\widehat{f}(\mymathbb{0};\widehat{Y})^\top Q^\top Q\widehat{x}+\frac{1}{2}\widehat{x}^\top Q^\top Q\nabla_{\mathrm{MN}}^2\widehat{f}(\mymathbb{0};\widehat{Y})Q^\top Q\widehat{x}\\
        &= \widehat{f}(\mymathbb{0})+\nabla_{\mathrm{MN}}\widehat{f}(\mymathbb{0};\widehat{Y})^\top\widehat{x}+\frac{1}{2}\widehat{x}^\top\nabla_{\mathrm{MN}}^2\widehat{f}(\mymathbb{0};\widehat{Y})\widehat{x}\\
        &= \widehat{m}_{\mathrm{MN}}^{\mymathbb{0},\widehat{Y}}(\widehat{x}).
    \end{align*}
    \end{proof}

    A special case of Theorem \ref{thm:full&subspaceMN} worth mentioning occurs when $\widehat{Y}$ is poised for quadratic interpolation.  In this case, the subspace MN model coincides with the subspace DQI model, i.e, $\widehat{m}_{\mathrm{MN}}^{\mymathbb{0},\widehat{Y}}(\widehat{x})=\widehat{m}_{\mathrm{DQI}}^{\mymathbb{0},\widehat{Y}}(\widehat{x})$.  The following corollary generalizes \cite[Theorem~3]{chen2024qfully}.  In particular, \cite[Theorem 3]{chen2024qfully} only shows that the full-space model corresponding to $\widehat{m}_{\mathrm{DQI}}^{\mymathbb{0},\widehat{Y}}(\widehat{x})$ is an underdetermined quadratic interpolation model of $f$ at $x^0$ over~$Y$, while the following corollary shows that it is actually the MN model of $f$ at $x^0$ over $Y$.
    \begin{coroll}
        Suppose that Assumptions~\ref{ass:Yisinsubspace} and \ref{ass:interpexists} hold.  Suppose that $\widehat{Y}$ is poised for quadratic interpolation.  Then,
        \begin{equation*}
            \left(\nabla_{\mathrm{MN}}f(x^0;Y), \nabla_{\mathrm{MN}}^2f(x^0;Y)\right) = \left(Q\nabla_{\mathrm{DQI}}\widehat{f}(\mymathbb{0};\widehat{Y}), Q\nabla_{\mathrm{DQI}}^2\widehat{f}(\mymathbb{0};\widehat{Y})Q^\top\right).
        \end{equation*}
        Moreover, 
        \begin{equation*}
            m_{\mathrm{MN}}^{x^0,Y}(x)=\widehat{m}_{\mathrm{DQI}}^{\mymathbb{0},\widehat{Y}}(\widehat{x})~~~\text{for all $(x,\widehat{x})\in\mathbb{R}^n\times\mathbb{R}^d$ with $x\in \left(x^0+Q\widehat{x}\right)+\mathrm{col}(Q)^\perp$}.
        \end{equation*}
    \end{coroll}

    We can prove results similar to Theorem \ref{thm:full&subspaceMN} for Hessian approximations arising from the MFN and LFU models.  As noted in Remarks \ref{rem:MFN_unique} and \ref{rem:LFU_unique} and illustrated in Example \ref{exp:MFNgradnotunique}, Assumption \ref{ass:interpexists} alone is not sufficient to ensure the uniqueness of the gradient approximation for either the MFN or LFU model.  Nonetheless, we characterize the exact relationship between the sets of all gradient approximations arising from the full-space and subspace MFN and LFU models.  Moreover, when their gradient approximations satisfy a simple relationship, we obtain model value relationships analogous to \eqref{eq:samemodelvalMN}.
    \begin{theo}\label{thm:full&subspaceMFN}
        Suppose that Assumptions~\ref{ass:Yisinsubspace} and \ref{ass:interpexists} hold.  Then, both $\nabla_{\mathrm{MFN}}^2f(x^0;Y)$ and $\nabla_{\mathrm{MFN}}^2\widehat{f}(\mymathbb{0};\widehat{Y})$ are unique and satisfy
        \begin{equation*}
            \nabla_{\mathrm{MFN}}^2f(x^0;Y) = Q\nabla_{\mathrm{MFN}}^2\widehat{f}(\mymathbb{0};\widehat{Y})Q^\top.
        \end{equation*}
        The sets $\mathcal{G}_{\mathrm{MFN}}^{x^0,Y}$ and  $\mathcal{G}_{\mathrm{MFN}}^{\mymathbb{0},\widehat{Y}}$ satisfy
        \begin{equation}\label{eq:MFN_gradrelation}
            \mathcal{G}_{\mathrm{MFN}}^{x^0,Y}=\left\{Q\alpha:\alpha\in\mathcal{G}_{\mathrm{MFN}}^{\mymathbb{0},\widehat{Y}}\right\} + \mathrm{col}(Q)^\perp.
        \end{equation}
        Moreover, for any pair of elements $(\nabla_{\mathrm{MFN}}f(x^0;Y),\nabla_{\mathrm{MFN}}\widehat{f}(\mymathbb{0};\widehat{Y}))\in\mathcal{G}_{\mathrm{MFN}}^{x^0,Y}\times\mathcal{G}_{\mathrm{MFN}}^{\mymathbb{0},\widehat{Y}}$ such that $\nabla_{\mathrm{MFN}}f(x^0;Y)-Q\nabla_{\mathrm{MFN}}\widehat{f}(\mymathbb{0};\widehat{Y})\in\mathrm{col}(Q)^\perp$, the corresponding models satisfy
        \begin{equation*}
            m_{\mathrm{MFN}}^{x^0,Y}(x)=\widehat{m}_{\mathrm{MFN}}^{\mymathbb{0},\widehat{Y}}(\widehat{x})~~~\text{for all $(x,\widehat{x})\in\mathbb{R}^n\times\mathbb{R}^d$ with $x=x^0+Q\widehat{x}$}.
        \end{equation*}
        In particular, if $\nabla_{\mathrm{MFN}}f(x^0;Y)=Q\nabla_{\mathrm{MFN}}\widehat{f}(\mymathbb{0};\widehat{Y})$, then
        \begin{equation*}
            m_{\mathrm{MFN}}^{x^0,Y}(x)=\widehat{m}_{\mathrm{MFN}}^{\mymathbb{0},\widehat{Y}}(\widehat{x})~~~\text{for all $(x,\widehat{x})\in\mathbb{R}^n\times\mathbb{R}^d$ with $x\in \left(x^0+Q\widehat{x}\right)+\mathrm{col}(Q)^\perp$}.
        \end{equation*}
    \end{theo}
    \begin{proof}
        Similar to the proof of Theorem \ref{thm:full&subspaceMN}, we can write Problem \eqref{pro:MFN} as
        \begin{align*}
            \begin{split}\label{pro:MFN'}
                \min_{\substack{\widehat{\alpha}\in\mathbb{R}^d,\widehat{H}\in S_d(\mathbb{R})\\ \widehat{\alpha}_{\perp}\in\mathbb{R}^{n-d}\\ \widehat{H}_0\in\mathbb{R}^{d\times (n-d)},\widehat{H}_{\perp}\in S_{n-d}(\mathbb{R})}}~~~&\frac{1}{2}\left\|\widehat{H}\right\|_F^2+\left\|\widehat{H}_0\right\|_F^2+\frac{1}{2}\left\|\widehat{H}_{\perp}\right\|_F^2\\
                s.t.~~~&\left(\widehat{d}^i\right)^\top \widehat{\alpha}+\frac{1}{2}\left(\widehat{d}^i\right)^\top \widehat{H}\left(\widehat{d}^i\right) = \widehat{f}(\widehat{d}^i)-\widehat{f}(\mymathbb{0}),~~~i=1,\ldots,m,
            \end{split}\tag{MFN'}
        \end{align*}
        which implies that $\widehat{H}_0$ and $\widehat{H}_{\perp}$ vanish at the optimal solution and so
        \begin{equation*}
            \nabla_{\mathrm{MFN}}^2f(x^0;Y) = Q\nabla_{\mathrm{MFN}}^2\widehat{f}(\mymathbb{0};\widehat{Y})Q^\top.
        \end{equation*}

        Next, from the deduction of Equation \eqref{eq:fullsubsameconstraint}, we have that $(\alpha,\nabla^2_{\mathrm{MFN}}f(x^0;Y))$ is feasible for Problem~\eqref{pro:MFN} if and only if $\alpha=Q\widehat{\alpha}+v$ where $(\widehat{\alpha},\nabla^2_{\mathrm{MFN}}\widehat{f}(\mymathbb{0};\widehat{Y}))$ satisfies the constraints of Problem~\eqref{pro:MFN'} and $v\in\mathrm{col}(Q)^\perp$.  Equation \eqref{eq:MFN_gradrelation} follows from the fact that $\alpha\in\mathcal{G}_{\mathrm{MFN}}^{x^0,Y}$ if and only if $(\alpha,\nabla^2_{\mathrm{MFN}}f(x^0;Y))$ is feasible for Problem~\eqref{pro:MFN}, and $\widehat{\alpha}\in\mathcal{G}_{\mathrm{MFN}}^{\mymathbb{0},\widehat{Y}}$ if and only if $(\widehat{\alpha},\nabla^2_{\mathrm{MFN}}\widehat{f}(\mymathbb{0};\widehat{Y}))$ satisfies the constraints of Problem~\eqref{pro:MFN'}.

        Finally, let $m_{\mathrm{MFN}}^{x^0,Y}$ and $\widehat{m}_{\mathrm{MFN}}^{\mymathbb{0},\widehat{Y}}$ be the models with $\nabla_{\mathrm{MFN}}f(x^0;Y)=Q\nabla_{\mathrm{MFN}}\widehat{f}(\mymathbb{0};\widehat{Y})+v$ for some $v\in\mathrm{col}(Q)^\perp$.  Let $x=x^0+Q\widehat{x}$.  We have
        \begin{align*}
            m_{\mathrm{MFN}}^{x^0,Y}(x) &= f(x^0)+\nabla_{\mathrm{MFN}}f(x^0;Y)^\top Q\widehat{x}+\frac{1}{2}\widehat{x}^\top Q^\top\nabla_{\mathrm{MFN}}^2f(x^0;Y)Q\widehat{x}\\
            &= \widehat{f}(\mymathbb{0})+\left(\nabla_{\mathrm{MFN}}\widehat{f}(\mymathbb{0};\widehat{Y})^\top Q^\top + v^\top\right) Q\widehat{x}+\frac{1}{2}\widehat{x}^\top Q^\top Q\nabla_{\mathrm{MFN}}^2\widehat{f}(\mymathbb{0};\widehat{Y})Q^\top Q\widehat{x}\\
            &= \widehat{f}(\mymathbb{0})+\nabla_{\mathrm{MFN}}\widehat{f}(\mymathbb{0};\widehat{Y})^\top Q^\top Q\widehat{x}+\frac{1}{2}\widehat{x}^\top Q^\top Q\nabla_{\mathrm{MFN}}^2\widehat{f}(\mymathbb{0};\widehat{Y})Q^\top Q\widehat{x}\\
            &= \widehat{f}(\mymathbb{0})+\nabla_{\mathrm{MFN}}\widehat{f}(\mymathbb{0};\widehat{Y})^\top\widehat{x}+\frac{1}{2}\widehat{x}^\top\nabla_{\mathrm{MFN}}^2\widehat{f}(\mymathbb{0};\widehat{Y})\widehat{x}\\
            &= \widehat{m}_{\mathrm{MFN}}^{\mymathbb{0},\widehat{Y}}(\widehat{x}).
        \end{align*}
        In particular, if $\nabla_{\mathrm{MFN}}f(x^0;Y)=Q\nabla_{\mathrm{MFN}}\widehat{f}(\mymathbb{0};\widehat{Y})$, then the result follows from the proof of~\eqref{eq:samemodelvalMN}.
    \end{proof}

    The results for the LFU models are similar to those for the MFN models.  A key distinction is the presence of the reference matrix $\overline{H}$, which can fix the model’s behavior outside $\mathcal{Y}$.  Consequently, while the full-space and subspace LFU models coincide on $\mathcal{Y}$, there is no general guarantee that they coincide along directions in $\mathrm{col}(Q)^\perp$ unless $\overline{H}=QQ^\top\overline{H}QQ^\top$.
    \begin{theo}\label{thm:full&subspaceLFU}
        Suppose that Assumptions~\ref{ass:Yisinsubspace} and \ref{ass:interpexists} hold.  Denote $\widehat{\overline{H}}=Q^\top\overline{H}Q\in S_d(\mathbb{R})$.  Then, both $\nabla_{\mathrm{LFU}}^2f(x^0;Y;\overline{H})$ and $\nabla_{\mathrm{LFU}}^2\widehat{f}(\mymathbb{0};\widehat{Y};\widehat{\overline{H}})$ are unique and satisfy
        \begin{equation*}
            \nabla_{\mathrm{LFU}}^2f(x^0;Y;\overline{H}) = Q\nabla_{\mathrm{LFU}}^2\widehat{f}(\mymathbb{0};\widehat{Y};\widehat{\overline{H}})Q^\top + \overline{H} - QQ^\top\overline{H}QQ^\top.
        \end{equation*}
        The sets $\mathcal{G}_{\mathrm{LFU}}^{x^0,Y,\overline{H}}$ and $\mathcal{G}_{\mathrm{LFU}}^{\mymathbb{0},\widehat{Y},\widehat{\overline{H}}}$ satisfy
        \begin{equation*}
            \mathcal{G}_{\mathrm{LFU}}^{x^0,Y,\overline{H}} = \left\{Q\alpha:\alpha\in\mathcal{G}_{\mathrm{LFU}}^{\mymathbb{0},\widehat{Y},\widehat{\overline{H}}}\right\} + \mathrm{col}(Q)^\perp.
        \end{equation*}
        Moreover, for any pair of elements $(\nabla_{\mathrm{LFU}}f(x^0;Y;\overline{H}),\nabla_{\mathrm{LFU}}\widehat{f}(\mymathbb{0};\widehat{Y};\widehat{\overline{H}}))\in\mathcal{G}_{\mathrm{LFU}}^{x^0,Y,\overline{H}}\times\mathcal{G}_{\mathrm{LFU}}^{\mymathbb{0},\widehat{Y},\widehat{\overline{H}}}$ such that $\nabla_{\mathrm{LFU}}f(x^0;Y;\overline{H})-Q\nabla_{\mathrm{LFU}}\widehat{f}(\mymathbb{0};\widehat{Y};\widehat{\overline{H}})\in\mathrm{col}(Q)^\perp$, the corresponding models satisfy
        \begin{equation*}
            m_{\mathrm{LFU}}^{x^0,Y,\overline{H}}(x)=\widehat{m}_{\mathrm{LFU}}^{\mymathbb{0},\widehat{Y},\widehat{\overline{H}}}(\widehat{x})~~~\text{for all $(x,\widehat{x})\in\mathbb{R}^n\times\mathbb{R}^d$ with $x=x^0+Q\widehat{x}$}.
        \end{equation*}
        In particular, if $\overline{H}=QQ^\top\overline{H}QQ^\top$ and $\nabla_{\mathrm{LFU}}f(x^0;Y;\overline{H})=Q\nabla_{\mathrm{LFU}}\widehat{f}(\mymathbb{0};\widehat{Y};\widehat{\overline{H}})$, then
        \begin{equation*}
            m_{\mathrm{LFU}}^{x^0,Y,\overline{H}}(x)=\widehat{m}_{\mathrm{LFU}}^{\mymathbb{0},\widehat{Y},\widehat{\overline{H}}}(\widehat{x})~~~\text{for all $(x,\widehat{x})\in\mathbb{R}^n\times\mathbb{R}^d$ with $x\in \left(x^0+Q\widehat{x}\right)+\mathrm{col}(Q)^\perp$}.
        \end{equation*}
    \end{theo}
    \begin{proof}
        Denote $\widehat{\overline{H}}_0=Q^\top \overline{H}Q_{\perp}\in\mathbb{R}^{d\times (n-d)}$ and $\widehat{\overline{H}}_{\perp}=Q_{\perp}^\top \overline{H}Q_{\perp}\in S_{n-d}(\mathbb{R})$.  Similar to the proof of Theorem \ref{thm:full&subspaceMN}, we can write Problem \eqref{pro:LFU} as
        \begin{align*}
            \min_{\substack{\widehat{\alpha}\in\mathbb{R}^d,\widehat{H}\in S_d(\mathbb{R})\\ \widehat{\alpha}_{\perp}\in\mathbb{R}^{n-d}\\ \widehat{H}_0\in\mathbb{R}^{d\times (n-d)},\widehat{H}_{\perp}\in S_{n-d}(\mathbb{R})}}~~~&\frac{1}{2}\left\|\widehat{H}-\widehat{\overline{H}}\right\|_F^2+\left\|\widehat{H}_0-\widehat{\overline{H}}_0\right\|_F^2+\frac{1}{2}\left\|\widehat{H}_{\perp}-\widehat{\overline{H}}_{\perp}\right\|_F^2\\
            s.t.~~~&\left(\widehat{d}^i\right)^\top \widehat{\alpha}+\frac{1}{2}\left(\widehat{d}^i\right)^\top \widehat{H}\left(\widehat{d}^i\right) = \widehat{f}(\widehat{d}^i)-\widehat{f}(\mymathbb{0}),~~~i=1,\ldots,m,
        \end{align*}
        which implies that $\widehat{H}_0=\widehat{\overline{H}}_0$ and $\widehat{H}_{\perp}=\widehat{\overline{H}}_{\perp}$ at the optimal solution and so
        \begin{align*}
            \nabla_{\mathrm{LFU}}^2f(x^0;Y;\overline{H}) &= Q\nabla_{\mathrm{LFU}}^2\widehat{f}(\mymathbb{0};\widehat{Y};\widehat{\overline{H}})Q^\top + QQ^\top\overline{H}Q_{\perp}Q_{\perp}^\top + Q_{\perp}Q_{\perp}^\top\overline{H}QQ^\top + Q_{\perp}Q_{\perp}^\top\overline{H}Q_{\perp}Q_{\perp}^\top\\
            &= Q\nabla_{\mathrm{LFU}}^2\widehat{f}(\mymathbb{0};\widehat{Y};\widehat{\overline{H}})Q^\top + \overline{H} - QQ^\top\overline{H}QQ^\top.
        \end{align*}
        Notice that 
        \begin{equation*}
            Q^\top\nabla_{\mathrm{LFU}}^2f(x^0;Y;\overline{H})Q = \nabla_{\mathrm{LFU}}^2\widehat{f}(\mymathbb{0};\widehat{Y};\widehat{\overline{H}}) + Q^\top\overline{H}Q - Q^\top\overline{H}Q = \nabla_{\mathrm{LFU}}^2\widehat{f}(\mymathbb{0};\widehat{Y};\widehat{\overline{H}})
        \end{equation*}
        and if $\overline{H}=QQ^\top\overline{H}QQ^\top$, then
        \begin{equation*}
            \nabla_{\mathrm{LFU}}^2f(x^0;Y;\overline{H}) = Q\nabla_{\mathrm{LFU}}^2\widehat{f}(\mymathbb{0};\widehat{Y};\widehat{\overline{H}})Q^\top.
        \end{equation*}
        The rest of the proof is the same as that of Theorem \ref{thm:full&subspaceMFN}.
    \end{proof}
    \begin{rema}
        If we take $\overline{H}$ to be the zero matrix, then Problem \eqref{pro:LFU} reduces to Problem \eqref{pro:MFN}, and Theorem \ref{thm:full&subspaceLFU} aligns with Theorem \ref{thm:full&subspaceMFN}. 
    \end{rema}

    We end this section with an example showing how $\overline{H}$ can fix the model's behavior outside $\mathcal{Y}$.
    \begin{exampl}
        Consider $n=3$, $f(x)=\|x\|^2$, $\overline{H}=I_3$, $x^0=\mymathbb{0}$ and $Y=\{\mymathbb{0},e^1,e^2,e^1+e^2\}\subseteq\mathbb{R}^3$.  Then, $d=2$, $Q=[e^1~e^2]\in\mathbb{R}^{3\times 2}$, $\widehat{f}(\widehat{x})=\|\widehat{x}\|^2$, $\widehat{\overline{H}}=I_2$, and $\widehat{Y}=\{\mymathbb{0},e^1,e^2,e^1+e^2\}\subseteq\mathbb{R}^2$.
        
        Solving Problem~\eqref{pro:LFU}, we get
        \begin{equation*}
            \mathcal{M}^{x^0,Y,\overline{H}}_{\mathrm{LFU}} = \left\{\left(\begin{bmatrix}
                \frac{1}{2}\\
                \frac{1}{2}\\
                r
            \end{bmatrix}, \begin{bmatrix}
                1 & 0 & 0\\
                0 & 1 & 0\\
                0 & 0 & 1
            \end{bmatrix}\right): r\in\mathbb{R}\right\}~~~\text{and}~~~\mathcal{M}^{\mymathbb{0},\widehat{Y},\widehat{\overline{H}}}_{\mathrm{LFU}} = \left\{\left(\begin{bmatrix}
                \frac{1}{2}\\
                \frac{1}{2}
            \end{bmatrix}, \begin{bmatrix}
                1 & 0\\
                0 & 1
            \end{bmatrix}\right)\right\}.
        \end{equation*}
        Now, pick 
        \begin{equation*}
            \nabla_{\mathrm{LFU}}f(x^0;Y;\overline{H}) = \begin{bmatrix}
                \frac{1}{2}\\
                \frac{1}{2}\\
                0
            \end{bmatrix}~~~\text{and}~~~\nabla_{\mathrm{LFU}}\widehat{f}(\mymathbb{0};\widehat{Y};\widehat{\overline{H}}) = \begin{bmatrix}
                \frac{1}{2}\\
                \frac{1}{2}
            \end{bmatrix}
        \end{equation*}
        so $\nabla_{\mathrm{LFU}}f(x^0;Y;\overline{H})=Q\nabla_{\mathrm{LFU}}\widehat{f}(\mymathbb{0};\widehat{Y};\widehat{\overline{H}})$.  The full-space and subspace LFU models are
        \begin{equation*}
            m_{\mathrm{LFU}}^{x^0,Y,\overline{H}}(x) = \begin{bmatrix}
                \frac{1}{2}\\
                \frac{1}{2}\\
                0
            \end{bmatrix}^\top x + \frac{1}{2}x^\top x~~~\text{and}~~~\widehat{m}_{\mathrm{LFU}}^{\mymathbb{0},\widehat{Y},\widehat{\overline{H}}}(\widehat{x}) = \begin{bmatrix}
                \frac{1}{2}\\
                \frac{1}{2}
            \end{bmatrix}^\top\widehat{x} + \frac{1}{2}\widehat{x}^\top\widehat{x}.
        \end{equation*}
        They coincide on $\mathcal{Y}=\{x\in\mathbb{R}^3:(e^3)^\top x=0\}$ but may not coincide along directions in $\mathrm{col}(Q)^\perp$.  Indeed, consider $x=e^3\in\mathbb{R}^3$ and $\widehat{x}=\mymathbb{0}\in\mathbb{R}^2$.  Then $x\in(x^0+Q\widehat{x})+\mathrm{col}(Q)^\perp$, but $m_{\mathrm{LFU}}^{x^0,Y,\overline{H}}(x)=\frac{1}{2}$, which differs from $\widehat{m}_{\mathrm{LFU}}^{\mymathbb{0},\widehat{Y},\widehat{\overline{H}}}(\widehat{x})=0$.
    \end{exampl}

\section{Full-space and subspace GSG, GSH, and QGSD models}\label{sec:full&subspaceGSGGSH}
    This section examines the relationships between the full-space and subspace GSG, GSH, and QGSD models.  We begin by presenting conversion formulas linking the full-space and subspace GSG and GSH.
    
    \begin{theo}\label{thm:full&subspaceGSG}
        Suppose that Assumption \ref{ass:Yisinsubspace} holds with
        \begin{equation*}
            Y=\{x^0\}\cup\{x^0+s^i:i=1,\ldots,p\}.
        \end{equation*}
        Then, there exist $\widehat{S}=[\widehat{s}^1\cdots\widehat{s}^p]\in\mathbb{R}^{d\times p}$ such that $S=Q\widehat{S}$.  The GSG of $f$ at $x^0$ over $S$ and the GSG of $\widehat{f}$ at $\mymathbb{0}$ over $\widehat{S}$ satisfy
        \begin{equation*}
            \nabla_sf(x^0;S) = Q\nabla_s\widehat{f}(\mymathbb{0};\widehat{S}).
        \end{equation*}
    \end{theo}
    \begin{proof}
        By definition, we have that
        \begin{equation*}
            \delta_f(x^0;S) = \begin{bmatrix}
                f(x^0+s^1) - f(x^0)\\
                \vdots\\
                f(x^0+s^p) - f(x^0)
            \end{bmatrix} = \begin{bmatrix}
                \widehat{f}(\widehat{s}^1) - \widehat{f}(\mymathbb{0})\\
                \vdots\\
                \widehat{f}(\widehat{s}^p) - \widehat{f}(\mymathbb{0})
            \end{bmatrix} = \delta_{\widehat{f}}(\mymathbb{0};\widehat{S}).
        \end{equation*}
        Since $Q$ has orthonormal columns, we have that $(S^\top)^\dagger = (\widehat{S}^\top Q^\top)^\dagger = Q(\widehat{S}^\top)^\dagger$ and so
        \begin{equation*}
            \nabla_sf(x^0;S) = \left(S^\top\right)^\dagger\delta_f(x^0;S) = Q\left(\widehat{S}^\top\right)^\dagger\delta_{\widehat{f}}(\mymathbb{0};\widehat{S}) = Q\nabla_s\widehat{f}(\mymathbb{0};\widehat{S}).
        \end{equation*}
    \end{proof}
    
    \begin{theo}\label{thm:full&subspaceGSH}
        Suppose that Assumption \ref{ass:Yisinsubspace} holds with
        \begin{align*}
            Y=\left\{x^0\right\}&\cup\left\{x^0+s^i:i=1,\ldots,p\right\}\\
            &\cup\left\{x^0+t_i^j:i=1,\ldots,p,j=1,\ldots,q_i\right\}\cup\left\{x^0+s^i+t_i^j:i=1,\ldots,p,j=1,\ldots,q_i\right\}.
        \end{align*}
        Then, there exist $\widehat{S}=[\widehat{s}^1\cdots\widehat{s}^p]\in\mathbb{R}^{d\times p}$ and $\widehat{T}_i=[\widehat{t}_i^1\cdots\widehat{t}_i^{q_i}]\in\mathbb{R}^{d\times {q_i}},i=1,\ldots,p$ such that $S=Q\widehat{S}$ and $T_i=Q\widehat{T}_i,i=1,\ldots,p$.  Denote $\widehat{T}_{1:p}=\{\widehat{T}_1,\ldots,\widehat{T}_p\}$.  The GSH of $f$ at $x^0$ over $S$ and $T_{1:p}$ and the GSH of $\widehat{f}$ at $\mymathbb{0}$ over $\widehat{S}$ and $\widehat{T}_{1:p}$ satisfy
        \begin{equation*}
            \nabla_s^2f(x^0;S;T_{1:p}) = Q\nabla_s^2\widehat{f}(\mymathbb{0};\widehat{S};\widehat{T}_{1:p})Q^\top.
        \end{equation*}
    \end{theo}
    \begin{proof}
        From Theorem \ref{thm:full&subspaceGSG}, we have that
        \begin{equation*}
            \delta_f^2(x^0;S;T_{1:p}) = \begin{bmatrix}
                \left(\nabla_s\widehat{f}(\widehat{s}^1;\widehat{T}_1)-\nabla_s\widehat{f}(\mymathbb{0};\widehat{T}_1)\right)^\top Q^\top\\
                \vdots\\
                \left(\nabla_s\widehat{f}(\widehat{s}^p;\widehat{T}_p)-\nabla_s\widehat{f}(\mymathbb{0};\widehat{T}_p)\right)^\top Q^\top
            \end{bmatrix} = \delta_{\widehat{f}}^2(\mymathbb{0};\widehat{S};\widehat{T}_{1:p})Q^\top
        \end{equation*}
        and so
        \begin{equation*}
            \nabla_s^2f(x^0;S;T_{1:p}) = \left(S^\top\right)^\dagger\delta_f^2(x^0;S;T_{1:p}) = Q\left(\widehat{S}^\top\right)^\dagger\delta_{\widehat{f}}^2(\mymathbb{0};\widehat{S};\widehat{T}_{1:p})Q^\top = Q\nabla_s^2\widehat{f}(\mymathbb{0};\widehat{S};\widehat{T}_{1:p})Q^\top.
        \end{equation*} 
    \end{proof}

    Using Theorems \ref{thm:full&subspaceGSG} and \ref{thm:full&subspaceGSH}, we establish model value results for the full-space and subspace QGSD models analogous to \eqref{eq:samemodelvalMN}.
    \begin{theo}
        Under the assumption and notation of Theorem \ref{thm:full&subspaceGSH}, the QGSD model of $f$ at $x^0$ over $Y$ and the QGSD model of $\widehat{f}$ at $\mymathbb{0}$ over $\widehat{Y}$ satisfy
        \begin{equation*}
            m_{\mathrm{QGSD}}^{x^0,Y}(x)=\widehat{m}_{\mathrm{QGSD}}^{\mymathbb{0},\widehat{Y}}(\widehat{x})~~~\text{for all $(x,\widehat{x})\in\mathbb{R}^n\times\mathbb{R}^d$ with $x\in \left(x^0+Q\widehat{x}\right)+\mathrm{col}(Q)^\perp$}.
        \end{equation*}
    \end{theo}
    \begin{proof}
        Notice that Theorems \ref{thm:full&subspaceGSG} and \ref{thm:full&subspaceGSH} apply to any linear combination of GSG and GSH, respectively. Therefore, we have that
        \begin{equation*}
            \nabla_{\mathrm{QGSD}}f(x^0;Y) = Q\nabla_{\mathrm{QGSD}}\widehat{f}(\mymathbb{0};\widehat{Y})~~~\text{and}~~~\nabla^2_{\mathrm{QGSD}}f(x^0;Y) = Q\nabla^2_{\mathrm{QGSD}}\widehat{f}(\mymathbb{0};\widehat{Y})Q^\top.
        \end{equation*}
        The result follows from the proof of \eqref{eq:samemodelvalMN}.
    \end{proof}

\section{Conclusion}\label{sec:conclusion}
    Motivated by recent developments in subspace model-based DFO methods, this paper develops a theoretical framework that clarifies the relationships between several widely used subspace modeling techniques and their full-space counterparts.  In particular, we analyze the connections between full-space and subspace MN models, MFN models, LFU models, GSG, GSH, and QGSD models.  For each modeling or approximation technique, we derive explicit conversion formulas between the full-space and subspace formulations.  We show that the full-space and subspace models coincide on the affine subspace and, in general, along directions in the orthogonal complement.  Overall, this paper develops theoretical tools for the analysis of subspace approximation techniques, with direct applicability to the design and analysis of model-based DFO methods.

    A natural next step is to apply the results in this paper to obtain new error bounds for these modeling and approximation techniques.   In particular, existing error bounds, which are typically derived under full-space sampling assumptions, can potentially be extended to settings where the sample set lies in a lower-dimensional subspace. Moreover, our framework may facilitate progress on the four open cases in which error bounds for the GSH remain unknown, as identified in \cite{hare2024matrix}.

    Another promising direction for future research is a detailed comparison of these modeling techniques, focusing on the conditions under which they coincide and under which one model may offer superior accuracy relative to others.  We note that this line of investigation is currently the subject of our ongoing work.

\section*{Declarations}
    \noindent\small{\textbf{Conflict of interest} The author declares that there are no competing interests.}

    \bigskip
    
    \noindent\small{\textbf{Data availability} This manuscript has no associated data.}

\normalsize
\bibliographystyle{siam}
\bibliography{references}
\end{document}